\documentclass[a4paper,12pt,final]{amsart}
\usepackage{amsaddr}
\usepackage[a4paper,left=24mm,right=24mm,top=34mm,bottom=34mm]{geometry}
\usepackage[utf8]{inputenc}
\usepackage{color}
\usepackage{xr-hyper} 
\usepackage{hyperref}
\usepackage{cleveref}
\usepackage{showlabels}
\usepackage{comment}
\usepackage[shortlabels]{enumitem}
\usepackage{todonotes}
\hypersetup{
	colorlinks,
	citecolor=red,
	filecolor=blue,
	linkcolor=blue,
	urlcolor=purple
}


\usepackage[all]{xy}
\usepackage{amsmath,amssymb,amsfonts,amsthm,graphicx}

\usepackage{multirow}
\usepackage{comment}
\theoremstyle{plain}

\tikzset{node distance=2cm, auto}

\newtheorem{thm}{Theorem}[section]
\newtheorem{conj}[thm]{Conjecture}
\newtheorem{cor}[thm]{Corollary}

\newtheorem{lm}[thm]{Lemma}

\theoremstyle{definition}
\newtheorem{df}[thm]{Definition}
\newtheorem{notation}[thm]{Notation}
\newtheorem{hyp}[thm]{Hypothesis}
\theoremstyle{remark}
\newtheorem{rem}[thm]{Remark}

\newcommand{\ZZ}{\operatorname{Z}}

\newcommand{\Alp}{\operatorname{Alp}}
\newcommand{\dz}{\operatorname{dz}}

\newcommand{\rdz}{\operatorname{rdz}}

\newcommand{\SL}{\operatorname{SL}}

\newcommand{\vphi}{\varphi}

\newcommand{\FF}{\ensuremath{\mathbb{F}}}

\newcommand{\NNN}{\operatorname{N}}

\newcommand{\GF}{ {\bG^F}}
\newcommand{\wGF}{ {\wbG^F}}
\newcommand{\wG}{ {\widetilde G }}
\newcommand{\Cent}{\operatorname{C}}
\newcommand{\Zent}{\operatorname{Z}}

\newcommand{\wt}{\widetilde}
\newcommand{\wbG}{{\widetilde \bG}}

\newcommand{\wN}{{\wt N}}

\usepackage{cleveref,enumitem}
\newlist{asslist}{enumerate}{1} % also creates a counter called 'propenumi'
\setlist[asslist]{label=(\roman*), ref=\thethm(\roman*)}

\crefalias{asslistenumi}{Assumption} 

\newlist{thmlist}{enumerate}{1} % also creates a counter called 'propenumi'
\setlist[thmlist]{label=(\alph*), ref=\thethm(\alph*)}

\newcommand{\calA}{{\mathcal A}}

\newcommand{\calB}{{\mathcal B}}

\newcommand{\calG}{\mathcal G}

\newcommand{\calP}{\mathcal P}

\newcommand{\bfG}{\bG}

\newcommand{\Bl}{\operatorname{Bl}}
\newcommand{\Out}{\operatorname{Out}}

\newcommand{\bl}{\operatorname{bl}}

\newcommand{\Irr}{\operatorname{Irr}}
\newcommand{\IBr}{\operatorname{IBr}}

\newcommand{\Aut}{\operatorname{Aut}}

\newcommand{\Lin}{\operatorname{Lin}}
\newcommand{\LinBr}{\operatorname{LinBr}}

\newcommand{\olQtheta}{{\ol{(Q,\theta)}}}
\newcommand{\olwtQwttheta}{{\ol{(\wt Q,\wt \theta)}}}

\newcommand{\ol}{\overline}

\newcommand{\wh}{\widehat}

\newcommand{\bG}{{{\mathbf G}}}

\numberwithin{equation}{section}

\excludecomment{commentMe}

\title{A criterion for the inductive Alperin weight condition}

\begin{document}
\begin{abstract}
We give a criterion that simplifies the checking of the inductive Alperin weight condition for the remaining open cases of simple groups of Lie type (see Theorem~\ref{AWStCond} below). It is strongly related in form to the criterion of the second author for the inductive McKay conditions (see \cite[2.12]{SpIMDefChar}) that has proved very useful. The proof follows from a Clifford theory for weights intrinsically present in the proof of reduction theorems of the Alperin weight conjecture given by Navarro--Tiep and the second author. We also give a related criterion for the inductive blockwise Alperin weight condition (Theorem~\ref{NewIndBAWCond}). 
\end{abstract}

\author{Julian Brough and Britta Sp\"ath}
\thanks{The first author thanks the Humbold Foundation for its support.}
\thanks{Both authors like to thank the research training group \emph{GRK 2240: Algebro-Geometric Methods in Algebra, Arithmetic and Topology}, funded by the DFG}
\email{brough@uni-wuppertal.de,\, bspaeth@uni-wuppertal.de}
\maketitle
\section{Introduction}

The Alperin weight conjecture and its blockwise version together form one of the deep remaining problems for the representation theory of finite groups.
These conjectures aim to count the irreducible modular representations of a group through certain local objects.
If $G$ is a finite group and $p$ a prime, then a $p$-{\it weight} of $G$ is a pair $(Q,\theta)$, where $Q$ is a $p$-subgroup of $G$ and $\theta\in \Irr(\NNN_G(Q)/Q)$ such that $\theta(1)_p=|\NNN_G(Q)/Q|_p$.

\begin{conj}\cite{AlpWei}
Let $G$ be a finite group and $p$ a prime. 
Then the number of irreducible $p$-Brauer characters of $G$ equals the number of $G$-conjugacy classes of $p$-weights of $G$.
\end{conj}

The blockwise Alperin weight conjecture forms a refinement of the previous conjecture by taking into account $p$-blocks of a finite group, which provide a partition of both the irreducible $p$-Brauer characters and $p$-weights.
  
\begin{conj}\cite{AlpWei}
Let $G$ be a finite group, $p$ a prime and $B$ a $p$-block of $G$. 
Then the number of irreducible $p$-Brauer characters of $B$ equals the number of $G$-conjugacy classes of $p$-weights of $B$.
\end{conj}

Since the statement of these conjectures, there have been multiple results towards proving them for large classes of finite groups. 
%In \cite{KnRobWeiConj}, Kn{\"o}rr and Robinson gave a reformulation using alternating sums and chains of $p$-subgroups.
%In addition Puig provided a refinement of the conjecture and reduced it to a statement about certain groups closely related to simple groups \cite{PuigFrobCat,PuigAlpWei}.

In \cite{NaTi_red}, Navarro and Tiep presented an approach which reduced the non-blockwise version to a certain set of conditions (the inductive AW condition) for every quasi-simple group.
The second author in \cite{BAWSp} refined the methods of Navarro and Tiep to provide a similar reduction for the blockwise version to the so-called inductive BAW condition.
Furthermore, these papers established the respective inductive condition for the case of simple groups of Lie type for $p$ the defining characteristic.

%Since its appearance the inductive BAW condition has been verified in multiple cases.
%In \cite{Breuer}, Breuer has dealt with most blocks for the sporadic simple groups.
%Malle gave a verification for the alternating, Suzuki and Ree groups in \cite{BAWMa}.
%Koshitani and the second author have shown the condition holds for all blocks which have cyclic defect groups when $p$ is odd \cite{KosSpAMBAWCy}. 
%The condition has also been verified for various simple groups of small rank in \cite{BrSchFrBAWSP4, FengLiLi, ASF_Sp6, Schulte} and Feng has considered unipotent blocks \cite{Z_Feng}.
%Furthermore, in \cite{LiZhang} and \cite{Li_Sp2nq}, the inductive blockwise Alperin weight condition has been studied for certain Lie types with an assumption on the outer-automorphism group and the decomposition matrix.

The inductive AW condition and BAW condition bear a strong resemblence to the inductive condition for the McKay conjecture and its blockwise counterpart the Alperin--McKay conjecture \cite{IMNRedMcKay,AMSp}.
Moreover, the second author has verified the inductive McKay condition for most of the families of groups of Lie type, using as a key ingredient Proposition 2.12 of \cite{SpIMDefChar}. We propose here an analogue of this statement for the inductive Alperin weight condition, see Theorem~\ref{AWStCond}.
Mimicking ideas of the reduction theorems given in \cite{NaTi_red} and \cite{BAWSp} we develop a kind of Clifford theory for weights between a group and a normal subgroup (see also \cite{CLMS} for an early use of projective representations in that context). This uses in an essential way properties proven in \cite{NavSpBrHtZero} about the Dade--Glauberman--Nagao correspondence.

During the preparation of this note various authors used our criterion already verifying the inductive Alperin weight condition and the inductive blockwise Alperin weight condition for families of simple groups of Lie type, see \cite{FeLiZh1,Li_Sp2nq2,FeLiZh2,FeMa}.

{\bf Acknowledgement:} We thank Marc Cabanes and Lucas Ruhstorfer for various insightful conversations. We are also indebted to Conghui Li and Gunter Malle for a thorough reading of an earlier version of this paper. 

\section{Preliminaries}

We start by recalling some results from \cite{NaTi_red} and \cite{BAWSp} on weights in general. We also recall the so-called Dade-Glauberman-Nagao correspondence (written as the DGN-correspondence). In the next part of this section we give some Clifford theory-like results for weights and conclude this section by highlighting some consequences of those results.

\subsection{Recall}

We use in general the notation of \cite{NavBl} for $p$-blocks, characters and $p$-Brauer characters. We assume a prime $p$ is fixed throughout the paper.

\begin{notation}
Let $G$ be a finite group. Then
$\Irr(G)$ denotes the set of irreducible characters of $G$, while $\Lin(G):=\{\chi\in\Irr(G)\mid \chi(1)=1\}$ is seen here as a multiplicative group.
Let $F$ be the field of characteristic $p$ defined as in \cite[p.~87]{NavBl}.
The set of irreducible $p$-Brauer characters of $G$ coming from representations over $F$ is denoted by $\IBr(G)$ (defined on $G^0:=G_{p'}$, the set of $p$-regular elements of $G$), while $\LinBr(G):=\{\varphi\in\IBr(G)\mid \varphi(1)=1\}$. Let $\Bl_p(G)$ denote the set of $p$-blocks of $G$. The $p$-block containing a given $\theta\in\IBr(G)$ is denoted by $\bl(\theta)\in\Bl_p(G)$.
Lastly ${\rm dz} (G)$ denotes the set of $p$-defect zero characters of $G$, that is
\[
{\rm dz} (G):=\{\chi\in \Irr(G)\mid \chi(1)_p=|G|_p\}.
\]
\end{notation}

The map $\chi\mapsto\chi^0$ induces a bijection between $\Lin(G)_{p'}$ and $\LinBr (G)$, 
where $\chi^0$ denotes the restriction of $\chi$ to $G^0$. 

\begin{notation}
Let $N\lhd G$ and $\chi\in \Irr(G)$.
Then $$\Irr(N\mid \chi):=\{\eta \in\Irr(N)\mid \langle\chi_N,\eta\rangle_N\neq 0\}.$$ Similarly for $\eta\in\Irr(N)$, one defines $\Irr(G\mid \eta):=\{\chi\in\Irr(G)\mid \langle\chi_N,\eta\rangle_N\neq 0\}$.
Define $\IBr(N\mid\varphi)$ and $\IBr(G\mid\theta)$ for $\varphi\in\IBr(G)$, $\theta\in\IBr(N)$, in an analogous way.
\end{notation}

\begin{df}
Let $G$ be a finite group and $p$ a prime. One denotes by
$\Alp ^0(G)$ the set of all $p$-weights of $G$, that is pairs $(Q,\theta)$ with $\theta\in{\rm dz} (\NNN_G(Q)/Q)$ for $Q$ a $p$-subgroup of $G$.
The $G$-orbit of $(Q,\theta)$ is denoted by $\ol{(Q,\theta)}$ and ${\Alp} (G):=\Alp^0 (G)/\sim_G$.

For $\nu\in\Lin(\ZZ(G))_{p'}$ we denote by $\Alp^0 (G\mid \nu)$ the pairs $(Q,\theta)\in \Alp^0 (G)$, such that $\theta$ lifts to a character in $\Irr(\NNN_G(Q) \mid \nu )$.
Note that this passes to the $G$-orbit and so $\Alp (G\mid \nu)$ is defined analagously.
\end{df}
%
%For the remainder of this paper the subscript $p$ from the above notation will be dropped with the assumption a prime $p$ has been fixed.

\begin{lm}
Let $G$ be a finite group.
Then $\Lin(G)_{p'}$ (and therefore $\LinBr(G)$) acts on $\Alp(G)$ by $\mu.(Q,\theta)=(Q,\theta.\ol\mu)$ where $\ol\mu\in \Lin({\NNN_G(Q)/Q})$ is the restriction of $\mu\in\Lin(G)_{p'}$ to ${\NNN_G(Q)}$ seen as a linear character with $Q$ in its kernel.
\end{lm}
\begin{proof}
If $Q$ is any $p$-subgroup, then $Q\leq \ker(\mu)$ as $\mu$ has $p'$-order.
Thus for $(Q,\theta)\in \Alp^0(G)$ the restriction of $\mu$ to $\NNN_G(Q)$ passes to the quotient $\NNN_G(Q)/Q$, say $\mu_{\NNN_G(Q)/Q}$.
Hence $\mu\cdot (Q,\theta):= (Q,\theta \mu_{\NNN_G(Q)/Q})$ is also a weight in $\Alp^0(G)$. 
The lemma now follows as this action commutes with $G$-conjugation. 
\end{proof}

\subsection{Clifford theory for weights}

Let us now recall a crucial ingredient, the Dade--Glauberman--Nagao correspondence (see \cite{NavSpJAlg}). 

The context is as follows.
	Let $K\lhd G$, $M/K$ a $p$-subgroup of $G/K$ and $b$ a $G$-invariant block of $K$ with defect group $D_0\leq \ZZ(M)$. 
	Let $D$ be a defect group of the unique block $B$ of $M$ covering $b$, so that $M=KD$ and $K\cap D=D_0$. Let $B'\in\Bl_p(\NNN_M(D))$ be the Brauer correspondent of $B$ and $b'$ the unique block of $\NNN_K(D)$ covered by $B'$. 
	Then there is an $\NNN_{G}(D)$-equivariant bijection $$\pi_D: \Irr(b)\longrightarrow \Irr(b')$$ according to \cite[Thm 5.2]{NavSpBrHtZero}. 
	More precisely \cite[Thm 5.2]{NavSpBrHtZero} defines $\pi_D$ only on the set $\Irr_D(b)$ of $D$-invariant characters in $\Irr(b)$, but $\Irr(b)$ and $\Irr(b')$ are in natural correspondence with $\Irr(D_0)$ (for example according to \cite[Thm.~9.12]{NavBl}) and hence 
	$\Irr(b)=\Irr_D(b)$ and $\Irr(b')=\Irr_D(b')$.
	This bijection $\pi_D$ is called the {\em Dade-Glauberman-Nagao (DGN) correspondence}.

%	Beyond that there is a bijection 
%	$\Delta_B:\Irr_0(B)\longrightarrow \Irr_0(B')$ such that corresponding characters satisfy \todo{tidy up}
	
%	, where $\geq_b$ is defined as in \cite{NavSpBrHtZero}.
%	This gives a $\Lin(\NNN_G(M))$-equivariant bijection 
%	\[\Delta_D: \Irr(\NNN_G(M)\mid \varphi) \longrightarrow \Irr(\NNN_G(D)\mid \Delta_B(\varphi)) \text{ for every }\varphi\in\Irr_0(B)
%	,\]
%	.   

When $G\lhd \wt G$ is a normal inclusion of finite groups then the weights of $G$ and $\wt G$ can be related via the DGN-correspondence. 
In this section we establish this relation and consider the Clifford theory for weights (see \Cref{CliffWei}).

Now assume $G\lhd \wt G$ and $(Q,\theta)\in\Alp^0(G)$. 
Set $N:=\NNN_G(Q)$, $\wt N:=\NNN_{\wt G}(Q)_\theta$.
Take $M/N$ a $p$-subgroup of $\wt N/N$.
For the unique block of $M/Q$ which covers the block $\bl(\theta)$ of $N/Q$ fix $D$ a defect group.
Then as above $D\cap (N/Q)=1$, $M/Q=(N/Q)D$ and $\NNN_{M/Q}(D)=\Cent_{N/Q}(D)D$, a direct product.
Note that $\pi_D(\theta)\in\dz(\Cent_{N/Q}(D))$.

\begin{df}
 Let $\ol \pi_D(\theta)$ be the character in $\Irr(\NNN_{M/Q}(D)/D)$ which lifts to $\pi_D(\theta)\times 1_D\in \Irr(\NNN_{M/Q}(D))$.
\end{df}

For $D=\wt Q/Q$ it follows that $M=N\wt Q$ and $\NNN_N(\wt Q)=\NNN_G(\wt Q)$.
Thus $\NNN_{M/Q}(D)/D\cong \NNN_G(\wt Q)\wt Q/\wt Q$ is normal in $\NNN_{\wt G}(\wt Q)/\wt Q$.

\begin{df}\label{DfCovWt}
With the above notation, $(\wt Q, \wt \theta)\in\Alp^0(\wt G)$ is said to {\it cover}  $(Q,\theta)\in \Alp^0(G)$ if $\wt \theta\in \dz(\NNN_{\wt G}(\wt Q)/\wt Q \mid \ol\pi_D(\theta))$.
We write $\Alp^0(\wt G\mid (Q,\theta))$ for the set of those weights.
\end{df}

Assume that $(\wt Q,\wt\theta)$ covers some $(Q,\theta)\in\Alp^0(G)$.
By Clifford theory and the equivariance of the map $\overline \pi_D$ we observe: $(\wt Q,\wt \theta)$ covers $(Q,\theta')$ if and only if $(Q,\theta')$ lies in the $\NNN_{\wt G}(\wt Q)$-orbit of $(Q,\theta)$. 
Moreover, if $g\in\wt G$, then $(\wt Q^g,\wt\theta^g)$ covers $(Q^g,\theta^g)$.
As an analogue, we say that $\olwtQwttheta\in \Alp(\wt G)$ covers $\ol{(Q,\theta)}\in\Alp(G)$ if $(\wt Q,\wt\theta)\in \Alp^0(\wt G\mid (Q,\theta)^g)$ for some $g\in \wt G$.
  
\begin{lm}\label{covWeight}
Let $G\lhd \wt G$. The elements in $\Alp(G)$ covered by $\olwtQwttheta\in \Alp(\wt G)$ form a non-empty $\wt G$-orbit.
\end{lm}  
  
\begin{rem}\label{CliffWei}
With the notation from \Cref{DfCovWt} this implies a Clifford-like partitioning of weights: 
	 \[\Alp(\wt G)= \bigsqcup_{\ol{(Q,\theta)}} \Alp(\wt G\mid \ol{(Q,\theta)}), \]
 where $\ol{(Q,\theta)}$ runs through a $\wt G$-transversal in $\Alp(G)$.
\end{rem}

It turns out that the sets $\Alp(\wt G\mid \ol{(Q,\theta)})$ are non-empty (see Corollary~\ref{wtYfromXabquo} below).

We require an application of a result from \cite{NaTi_red} about the consequences of DGN-correspondence.
Let us first recall the following notation:
\begin{notation}
If $G\lhd \wt G$ and $\chi\in \Irr(G)$ we set
\[
{\rm rdz}(\wt G\mid \chi):= \left\{ \wt \chi\in \Irr(\wt G\mid\chi)\ \big| \ \left(\frac{\wt\chi(1)}{\chi(1)}\right)_p=|\wt G/G|_p\right\},
\]
the set of so-called relative $p$-defect zero characters.
\end{notation}

\begin{thm}\label{DZBijRDZ}
Let $G\lhd \wt G$, $(Q,\theta)\in \Alp^0(G)$, $N:=\NNN_G(Q)$, $\wt N:=\NNN_{\wt G}(Q)_\theta$ and let $M/N$ be a $p$-subgroup of $\wt N/N$. 
Set $\wt Q/Q$ to be a defect group of the unique block of $M/Q$ covering $bl(\theta)$ and $\pi_{\wt Q/Q}(\theta) \in \Irr(\NNN_{N/Q}(\wt Q/Q))$ the DGN-correspondent. 

There exists a $\NNN_G(M)$-invariant extension $\wh \theta$ of $\theta$ to $M/Q$ and a $\Lin(\NNN_{\wt N}(M)/M)$- equivariant bijection 
\[\Delta_\theta:\rdz(\NNN_{\wt N}( M)\mid \wh \theta) \longrightarrow \dz(\NNN_{\wt G} (\wt Q)/\wt Q\mid \ol\pi_{\wt Q/Q}(\theta)),\] 
where $\bl(\varphi)=\bl(\varphi')^{\NNN_{\wt N}(M)}$ for every $\varphi\in \rdz(\NNN_{\wt N}(M)| \wh \theta)$, whenever $\varphi'\in\Irr(\NNN_{\wt G}(\wt Q))$ is the lift of $\Delta_\theta(\varphi)$.
\end{thm}
\begin{proof}
	This is essentially \cite[Thm~4.2]{NaTi_red} and \cite[Thm~3.8]{BAWSp}.  
	Note that for the statement about block induction we apply \cite[Cor.~2.11]{BAWSp}. 

By \cite[Thm~3.8]{BAWSp} there is a bjiection
\[
\Delta_\theta:\rdz(\NNN_{\wt N}( M)/Q \mid \wh \theta) \longrightarrow \dz(\NNN_{\NNN_{\wt N}( M)} (\wt Q)/\wt Q \mid \ol\pi_{\wt Q/Q}(\theta))
\]
However $\NNN_{\wt G}(\wt Q)\leq \NNN_{\wt G}(M)\cap \NNN_{\wt G}(Q)$ and thus
\[
\NNN_{\NNN_{\wt N}( M)} (\wt Q)/\wt Q=\left( \NNN_{\NNN_{\wt G}(M)\cap \NNN_{\wt G}(Q)} (\wt Q)/\wt Q \right)_\theta=\NNN_{\wt G}(\wt Q)_\theta/\wt Q.
\]

Both $\NNN_G(Q)$ and $\NNN_G(\wt Q)$ are $\NNN_{\wt G}(\wt Q)$-stable and thus the equivariance of the DGN-correspondence (see \cite[5.2]{NavSpBrHtZero}) shows that
\[
\NNN_{\wt G}(\wt Q)_\theta=\NNN_{\wt G}(\wt Q)_{\pi_{\wt Q/Q}(\theta)}=\NNN_{\wt G}(\wt Q)_{\ol\pi_{\wt Q/Q}(\theta)}
\]
Hence Clifford correspondence, given by character induction, provides a bijection 
\[
\dz \left(\NNN_{\NNN_{\wt N}( M)}(\wt Q)/\wt Q \mid  \ol\pi_{\wt Q/Q}(\theta)\right) \leftrightarrow {\rm dz}\left( \NNN_{\wt G}(\wt Q)/\wt Q\mid \ol\pi_{\wt Q/Q}(\theta) \right).
\qedhere
\]

\end{proof}

\begin{rem}\label{rem28}
		 The construction of the (non-unique) map $\Delta_\theta$ is $\wG$-equivariant in the sense that $(\Delta_\theta)^x$ gives a map $\Delta_{\theta^x}$ with the above properties. 
\end{rem}

\begin{cor}\label{wtYfromXabquo}
Let $(Q,\theta)\in\Alp^0(G)$ and $G\lhd \wt G$.
Then $\Alp^0(\wt G\mid (Q,\theta))$ is non-empty.
\begin{proof} (See also \cite[5.10]{BAWSp} and its proof.)
By the classical construction (see \cite[Ch. 11]{IsaChTh}) the invariant character $\theta$ gives rise to a central extension $K$ of $\NNN_{\wt G}(Q)_{\theta}/\NNN_G(Q)$, with the quotient map $\epsilon: K\longrightarrow \NNN_{\wt G}(Q)_{\theta}/\NNN_G(Q)$ of kernel a central $p'$-group $Z$, together with a faithful $\nu\in \Irr(Z)$.
Any weight $(M,\kappa)$ of $K$ above $\nu$ corresponds via $\epsilon$ to some $\kappa'\in\rdz(H\mid \wt \theta)$, where $H/\NNN_G(Q)=\epsilon(\NNN_K(M))$. Via $\Delta_\theta$ from \Cref{DZBijRDZ}, $\kappa'$ defines a weight of $\wt G$ above $(Q,\theta)$.
\end{proof}
\end{cor}
%\newpage

\subsection{Clifford theory of weights in a particular situation}
In this section we assume that $G$ and $\wt G$ from before satisfy additional properties, that hold for example when $\wt G/G$ is cyclic (see \Cref{Hyp:WeiBrExt}).
For applications in Section~3, we describe the structure of stabilisers of weights and Brauer characters under the following assumptions. 
	
\begin{hyp}\label{Hyp:WeiBrExt}
Let $G\lhd \wt G$ with abelian $\wt G/G$ such that 
	\begin{enumerate}
\item every $\varphi\in\IBr(G)$ extends to $\wG_\varphi$, and
\item for every $\ol{(Q,\theta)}\in \Alp(G)$ the character $\theta$ extends to $(\NNN_{\wt G}(Q)/Q)_\theta$.
	\end{enumerate} 

\noindent According to \Cref{DZBijRDZ}, Hypothesis \ref{Hyp:WeiBrExt}(2) is equivalent to:
	\begin{enumerate}
\item[(2')] for every $\ol{(\wt Q, \wt\theta)}\in\Alp(\wt G)$, any $\theta_0\in\IBr({\NNN_G(\wt Q)\wt Q/\wt Q}\mid \wt\theta)$ extends to $(\NNN_{\wt G}(\wt Q)/\wt Q)_{\theta_0}$.
\end{enumerate}
\end{hyp}

In \cite{SpIMDefChar} the following observation for ordinary characters was exploited: 
If $\chi\in \Irr(G)$ extends to $\wt G_\chi$, $\wt G/G$ is abelian and $\wt  \chi\in\Irr(\wG\mid \chi)$, 
Clifford theory shows that $$\wG_\chi=\bigcap_{\mu\in \Irr(\wG/G)_{\wt \chi}}\ker(\mu)$$ and $\Irr(\wG\mid \chi)$ is a $\Lin(\wG/G)$-orbit in $\Irr(\wG)$. We give an analogous description for Brauer characters. 

\begin{lm}\label{lem216}
Let $\varphi\in \IBr(G)$, $\wt \varphi\in\IBr(\wG\mid \varphi)$ and assume \Cref{Hyp:WeiBrExt}. 
Set ${\rm J}_G(\wt \varphi):=\wG_\varphi \wG_{p'} $, where $\wG_{p'}/G$ is the Hall $p'$-subgroup of $\wG/G$. Then 
\begin{enumerate}[(a)]
\item $\IBr(\wG\mid \varphi)$ is a $\LinBr(\wG/G)$-orbit,  and
\item \[\wG_{\varphi}=\left( \bigcap_{\mu\in\LinBr(\wG/G)_{\wt \varphi} }\ker(\mu)\right) \cap {\rm J}_G(\wt\varphi).\] 
\end{enumerate}
\end{lm}
\begin{proof}
Let $\wt \varphi_0\in \IBr(\wt G_{\varphi}\mid \varphi)$ with $\wt \varphi_0^{\wt G}=\wt \varphi$. Then $\wt\varphi_0$ is an extension of $\varphi$ by Hypothesis~\ref{Hyp:WeiBrExt}.
According to \cite[8.20]{NavBl} we have $\IBr(\wt G\mid \varphi)=\{ (\wt \varphi_0 \mu')^{\wG}  \mid \mu' \in \IBr(\wG_\varphi /G)\}$. This proves part (a). 

Because of  $\wt \varphi_0^{\wt G}=\wt \varphi$ we see that $\wG_\varphi\leq \bigcap_{\mu\in\LinBr(\wG/G)_{\wt \varphi} }\ker(\mu)$. 
By definition ${\rm J}_G(\wt\varphi)$ satisfies $\wG_{\varphi}\leq {\rm J}_G(\wt\varphi)$ and $ p\nmid  |{\rm J}_G(\wt\varphi):\wG_{\varphi}|  $. Since $| \bigcap_{\mu\in\LinBr(\wG/G)_{\wt \varphi} }\ker(\mu):\wG_\varphi|$ is a power of $p$, this implies (b). 
\end{proof}

The above statement has an analogue for a weight $\olQtheta\in\Alp(G)$, i.e. $\wG_{\olQtheta}$  and $\Alp(\wG\mid \olQtheta)$ can be described in a similar way.

\begin{lm}\label{lem29}
	Let $\ol{(Q,\theta)}\in\Alp(G)$, $\ol{(\wt Q,\wt \theta)}\in\Alp(\wt G\mid \ol{(Q,\theta)})$ and assume \Cref{Hyp:WeiBrExt}. 
Set	${\rm J}_{G}(\ol{(\wt Q, \wt \theta)}):=\wG_{\ol{(Q,\theta)}} \wG_{p'} $, where $\wG_{p'}/G$ is the Hall $p'$-subgroup of $\wG/G$. Denote  $N:=\NNN_G(Q)$ and $\wt N:=\NNN_{\wt G}(Q)_\theta$.
	Then 
	\begin{enumerate}[(a)]
\item $\Alp(\wG\mid \ol{(Q,\theta)})$ forms an $\LinBr(\wt G/G)$-orbit in $\Alp(\wG)$, 
and
\item $$\wG_ {\ol{(Q,\theta)}}= \left( \bigcap_{\mu\in \LinBr(\wt G/G)_{\olwtQwttheta}} \ker(\mu)\right) \cap {\rm J}_G(\olwtQwttheta).$$ 
	\end{enumerate}
\end{lm}
\begin{proof} Let $M:=\wt Q N$.
	Note that $\wt N/N\cong G\wt N/G$ is abelian and by Hypothesis \ref{Hyp:WeiBrExt} all characters of $\Irr(\wt N/Q\mid \theta)$ are extensions of $\theta$.  
	Let $\Delta_\theta$ be the map from Theorem~\ref{DZBijRDZ} for the extension  $\wh \theta$ of $\theta$ to $M/Q$. 
	Then the character $\wt \theta$ corresponds to some character in $\rdz(\NNN_\wN(M)\mid \wh \theta  )$ via the bijection $\Delta_\theta$. The set $\rdz(\NNN_\wN(M)\mid \wh \theta  )$ is only non-empty if $M/N$ is the Sylow $p$-subgroup of $\wt N/N$.
	
	Note that all extensions of $\theta$ to $M/Q$ lie in the same block and have
	$\wt Q/Q$ as a defect group. Hence every weight in $\Alp(\wG\mid \overline {(Q,\theta)} )$ is of the form $\overline{(\wt Q, \wt \theta')}$ 
	for some $\wt \theta'\in \dz(\NNN_\wN(\wt Q)/\wt Q \mid \overline \pi_{\wt Q/Q}(\theta))$. 
	The bijection $\Delta_\theta$ from \Cref{DZBijRDZ} is $\Lin(\wt N/M)$-compatible and $\rdz(\NNN_{\wt N}(M)\mid \wh \theta)$ forms a $\Lin(\wt N/M)$-orbit. Accordingly $\dz(\NNN_\wN(\wt Q)/\wt Q \mid \overline \pi_{\wt Q/Q}(\theta))$ is a $\Lin(\wt N/M)$-orbit as well. This proves part (a). 
	
 	For part (b) we apply the arguments from Lemma \ref{lem216}(b) to $\wt G_{\overline{( Q,\theta)}}=G \NNN_\wG( Q)_\theta= G \wt N$ and use that $\Delta_\theta$ is $\Lin(\wt N/M)$-compatible.
\end{proof}

\begin{lm}\label{WeightCorrYtoX}
Assume $G\lhd \wt G$ satisfies Hypothesis~\ref{Hyp:WeiBrExt}.
	Then there is a bijection 
	\[ \Pi: \Alp(\wt G)/_{\sim\LinBr(\wt G/G)} \longrightarrow \Alp(G)/_{\sim_{\wt G/G}},\] 
	which sends each $\LinBr(\wt G/G)$-orbit containing $\ol{(\wt Q,\wt \theta)}$ to the $\wt G$-orbit of weights of $G$ covered by $\ol{(\wt Q,\wt \theta)}$.
\end{lm}
\begin{proof}
For every $\overline{(Q,\theta)}\in\Alp(G)$ we define $\Pi$ by mapping any element of $\Alp(\wG \mid \ol{(Q,\theta)})$ to the $\wG$-orbit containing  $\ol{(Q,\theta)}$.  According to \Cref{covWeight} and \Cref{lem29} this map is well-defined and injective. Following \Cref{wtYfromXabquo} the map is surjective. 
\end{proof}

%\newpage

\section{The inductive AW condition}

The following definition is equivalent to the one stated in \cite[\S 3]{NaTi_red} (see also \cite{KosSpAMBAWCy}), although the version presented here uses the notion of weight, instead of sets $\IBr(G\mid Q)$.

\begin{df}\label{IndAWCond}
Let $p$ be a prime, $S$ a finite non-abelian simple group, $G$ a $p'$-covering group of $S$ (that is a maximal perfect group of which $S$ is a central quotient by a $p'$-group).
We say that {\bf the inductive AW condition holds for $S$ and $p$} if the following statements are satisfied: 
\begin{asslist}

\item\label{indAWiii}
There exists an $\Aut(G)$-equivariant bijection $$\Omega:\IBr(G)\rightarrow \Alp(G)$$ such that $\Omega(\IBr(G\mid \nu))=\Alp(G\mid\nu)$ for every $\nu\in\IBr(\ZZ(G))$.

\item For every $\varphi\in \IBr(G)$ and $\ol{(Q,\theta)}=\Omega(\varphi)$, there exists a finite group $A:=A(\varphi,Q)$ and characters $\wt\varphi\in\IBr(A)$ and $\wt\theta\in\IBr(\NNN_A(Q))$ such that
\begin{enumerate}
\item The group $A$ satisfies $G\lhd A$, $A/{\rm C}_A(G)\cong \Aut(G)_\varphi$, ${\rm C}_A(G)=Z(A)$ and $p\nmid |Z(A)|$.
\item $\wt\varphi\in\IBr(A)$ is an extension of $\varphi$.
\item $\wt\theta\in\IBr(\NNN_A(\ol Q))$ is an extension of $\Omega(\varphi)^0$.
\item $\IBr({\ZZ(A)}\mid\wt \varphi)=\IBr( {\ZZ(A)}\mid\wt \theta)$.
\end{enumerate}
\end{asslist}
\end{df} 

\begin{rem}\label{CondiiiProjRep}
In \cite{SpIndCharTrip} the inductive conditions for some of the local-global conjectures were rephrased in terms of character triples.
In particular, by \cite[Thm 4.3]{SpIndCharTrip}, the requirement on $(\varphi,\Omega(\varphi))$ in \Cref{IndAWCond}(ii) is equivalent to showing that 
\[
(G\rtimes \Aut(G)_{\varphi},G,\varphi)\succ_{F,c} ((G\rtimes \Aut(G))_{Q,\theta}, \NNN_G(Q),\theta^0)
\] in the notation of \cite[3.1]{SpIndCharTrip}.

If, in addition $ G\rtimes \Aut(G)_{\varphi}=G( G\rtimes \Aut(G))_{(Q,\theta)}$ then \cite[Proposition 3.7]{SpIndCharTrip} proves this property of character triples has an equivalent formulation in terms of projective representations.
Namely, there exists two modular projective representations $\mathcal{P}$ and $\mathcal{P}'$ of $G\rtimes \Aut(G)_{\varphi}$ and $(G\rtimes \Aut(G))_{Q,\theta}$ associated with $\varphi$ and $\theta$ in the sense of \cite[\S3]{SpIndCharTrip} such that 
\begin{itemize}
\item the factor sets of $\mathcal{P}$ and $\mathcal{P}'$ coincide on $(G\rtimes \Aut(G))_{Q,\theta}\times (G\rtimes \Aut(G))_{Q,\theta}$, and
\item for $x\in \Cent_{G\rtimes \Aut(G)_{\varphi}}(G)$ the matrices $\mathcal{P}(x)$ and $\mathcal{P}'(x)$ are associated with the same scalar.
\end{itemize}
\end{rem} 
 
 We now state our main result.
 
\begin{thm}\label{AWStCond}
Let $S$ be a finite non-abelian simple group and $p$ a prime dividing $|S|$. Let $G$ be a $p'$-covering group of $S$.
Assume we have a semi-direct product $\wt G \rtimes E$, such that the following conditions hold:

\begin{asslist}
\item 
\begin{itemize}[label={\textbullet}]
\item $G=[\wt G,\wt G]$,
\item $\Cent_{\wt G\rtimes E}(G)=\Zent(\wt G)$ and $\wt GE/\Zent(\wt G)\cong \Aut (G)$ by the natural map,
\item every $\vphi_0\in\IBr(G)$ extends to $\wt G_{\vphi_0}$,
\item for every $\ol{(Q,\theta_0)}\in\Alp(G)$ the character $\theta_0$ extends to  $(\NNN_{\wt G}(Q)/Q)_{\theta_0}$.
\end{itemize}

\item There exists some $\LinBr(\wt G/G)\rtimes E$-equivariant bijection 
$\wt \Omega: \IBr(\wt G)\longrightarrow \Alp(\wt G)$ such that 
\begin{enumerate}
	\item 	$\wt \Omega(\IBr(\wt G\mid \nu ^0))=\Alp(\wt G\mid \nu)$ for every $\nu\in\Lin(\ZZ(\wt G))_{p'}$, and
	\item 
	  $\rm J_{G}(\wt\varphi)=\rm J_{G}(\wt \Omega(\wt\varphi))$ for every $\wt\varphi\in\IBr(\wt G)$ (using the notation of Lemmas~\ref{lem216} and~\ref{lem29} above).
\end{enumerate}

\item For every $\wt{\varphi}\in \IBr(\wt G)$ there exists some $\varphi\in \IBr (G\mid \wt{\varphi})$ such that  
\begin{itemize}[label={\textbullet}]
\item $(\wt G\rtimes E)_{\varphi}=\wt G_{\varphi}\rtimes E_{\varphi}$, and 
\item $\varphi$ extends to $G\rtimes E_{\varphi}$.
\end{itemize}

\item In every $\wt G$-orbit on $\Alp(G)$ there exists a weight $\ol{(Q,\theta)}$ such that
\begin{itemize}[label={\textbullet}]
\item $(\wG E)_{\ol{(Q,\theta)}}=\wG_{\ol{(Q,\theta)}}E_{\ol{(Q,\theta)}}$, and
\item $\theta$ extends to $(GE)_{Q,\theta}/Q$.
\end{itemize} 
\end{asslist}
Then the inductive AW condition from \cite[Section 3]{NaTi_red} holds for $S$ and $p$.
\end{thm} 

\begin{lm}\label{CohomPairs}
Let $\varphi\in\IBr(G)$ and $\ol{(Q,\theta)}\in\Alp(G)$ such that the following holds:
\begin{enumerate}[(i)]
\item $(\wt G\rtimes E)_{\varphi}=\wt G_{\varphi}\rtimes E_{\varphi}$ and $\varphi$ extends to $G\rtimes E_\varphi$ and to $\wt G_\varphi$,
\item $(\wG E)_{\overline{(Q,\theta)}}=\wG_{\overline{(Q,\theta)}} E_{\overline{(Q,\theta)}}$ and $\theta$ extends to $(G\rtimes E)_{Q,\theta}/Q$ and to $\NNN_{\wt G}(Q)_\theta/Q$,
\item $(\wG E)_\varphi=(\wG E)_{\ol{(Q,\theta)}}$, and
\item there exist $\wt \varphi\in\IBr(\wG \mid \varphi)$ and $\olwtQwttheta\in\Alp(\wG\mid \olQtheta)$ such that 
  there exists for every $x\in E_{\varphi}$ some $\mu\in\IBr(\wt G/G)$ with 
  \[ \mu\cdot\olwtQwttheta=\olwtQwttheta ^x \text{ and  } \wt \varphi^x=\wt\varphi\mu, \]
and $\IBr( {\ZZ(\wt G)}\mid \wt \varphi)= \IBr({\ZZ(\wt G)}\mid \wt \theta^0)$.
\end{enumerate}
Let $Z:=\ker(\varphi)\cap \ZZ(G)$. Then
\[
\left( (\wt GE)_\varphi/Z,G/Z,\ol\varphi\right)\succ_{F,c} \left( (\wt GE)_{Q,\theta}/Z,\NNN_G(Q)/Z,\theta^0\right).
\]
\begin{proof}
As in the proof of \cite[Lem.~ 2.13]{SpIMDefChar} the assumptions allow the  construction of modular projective representations $\calP$ and $\calP'$ of $(\wt G/Z E)_\varphi$ and $(\wt G/Z E)_{Q,\theta}$ associated with the Brauer characters $\vphi $ and $\theta^0$ whose factor sets agree on $(\wt G/Z E)_{Q,\theta}\times (\wt G/Z E)_{Q,\theta}$ and such that for every $x\in\Cent_{(\wt G/Z  E)_\varphi}(G/Z)$ the matrices $\calP(x)$ and $\calP'(x)$ are scalar matrices associated with the same scalar. 
\end{proof}
\end{lm}

\begin{proof}[Proof of \Cref{AWStCond}]
Let $\mathcal{A}$ be an $\LinBr(\wt G/G)\rtimes E$-transversal in ${\Alp}(\wt G)$. 
By the assumption (i) and \Cref{lem29}, each $\LinBr(\wG/G)$-orbit in $\Alp(\wG)$ coincides with a set $\Alp(\wG\mid \ol{(Q,\theta)})$ for some $\ol{(Q,\theta)}\in\Alp(G)$.
For each weight $\ol {(\wt Q, \wt\theta)}\in\calA$ we find according to \Cref{covWeight} a weight $\ol{(Q,\theta)}$ of $G$ that is covered by it and by assumption (iv) we can choose it such that $(\wG E)_{Q,\theta}=\wG_{Q,\theta}(GE)_{Q,\theta}$ and $\theta$ extends to $(GE)_{Q,\theta}$. Let $\calA_0$ be the set of $\ol{(Q,\theta)}$ constructed this way, which forms a $\wG\rtimes E$-transversal in $\Alp(G)$, see \Cref{rem28} and \Cref{lem29}.

Since $\wt \Omega$ is $\LinBr(\wG/G)\rtimes E$-equivariant by (ii) the set  $\mathcal{G}:=\wt \Omega^{-1}(\calA)$ is a $\LinBr(\wG/G)\rtimes E$-transversal in $\IBr(\wG)$. 
For each $\wt\varphi\in \mathcal{G}$ we fix some $\varphi_0\in \IBr(G\mid \wt \varphi)$ satisfying Condition (iv) of the theorem and let $\mathcal{G}_0$ denote the set of such $\varphi_0$. By standard Clifford theory $\calG_0$ is a $\wG\rtimes E$-transversal in $\IBr(G)$.

Fix $\wt\varphi\in\calG$. Let $\ol{(\wt Q,\wt\theta)}:=\wt \Omega(\wt\varphi)\in\calA$,  
$\varphi_0\in\calG_0$ be  covered by $\wt\varphi$ and
$\ol{(Q,\theta)}\in\calA_0$ be covered by $\ol{(\wt Q,\wt\theta)}\in \calA$. Then $\varphi_0$ and $\ol{(Q,\theta)}$ are uniquely defined. 
We set $\Omega(\varphi_0)=\ol{(Q,\theta)}$.

Then according to \Cref{lem216},  $\wt G_\varphi=\left (\bigcap_{\mu\in \LinBr(\wG/G)_{\wt\varphi}} \ker(\mu)\right ) \cap \rm J_G(\wt \varphi)$. 
Analogously $\wG_{\ol{(Q,\theta)}}=\left (\bigcap_{\mu\in \LinBr(\wt N/N)_{\ol{(\wt Q,\wt \theta)}}} \ker(\mu)\right )\cap \rm J_{G}(\ol{\wt Q,\wt \theta})$, see \Cref{lem29}. By using the $\LinBr(\wt G/G)$-equivariance of $\wt \Omega$ and the assumption $\rm J_{G}(\wt\varphi)= \rm J_{G}(\ol{(\wt Q, \wt\theta )})$ we see  $\wG_{\varphi_0}=\wG_{\ol{(Q,\theta)}}$. 

Recall $\varphi_0\in\IBr(G)$ with $(\wG E)_{\varphi_0}=\wG_{\varphi_0} E_{\varphi_0}$.
This means that the $\wG$-orbit and the $E$-orbit of $\varphi_0$ intersect only at $\vphi_0$, thus implying easily  \[E_{\varphi_0}= E_{\wG\text{-orbit}(\varphi_0)}=
E_{\IBr(\wG\mid \varphi_0)}.\]
Similarly for $\ol{(Q,\theta)}\in\Alp(G)$ we see \[E_{\ol{(Q,\theta)}}= E_{\wG\text{-orbit}(\ol{(Q,\theta)})}=
E_{\Alp(\wG\mid \ol{(Q,\theta)})}\] according to \Cref{rem28}. Then $\wt \Omega$ maps  $\IBr(\wG\mid \varphi_0)$, the $\LinBr(\wG/G)$-orbit containing $\wt\varphi$, to  $\Alp(\wG\mid \ol{(Q,\theta)})$, the $\LinBr(\wG/G)$-orbit containing $\wt\Omega(\wt\varphi)$. Since $\wt \Omega$ is $E$-equivariant this implies \[E_{\varphi_0}=E_{\IBr(\wG\mid \varphi_0)}=
E_{\Alp(\wG\mid \ol{(Q,\theta)})}=E_{\ol{(Q,\theta)}}.
\]

This proves  
\[ (\wt G\rtimes E)_{\varphi_0}=(\wt G\rtimes E)_{\ol{(Q,\theta)}}\] 
and hence there exists a unique $\wG\rtimes E$-equivariant bijection  $\Omega:\IBr(G)\rightarrow \Alp(G)$ with the given values on $\calG_0$. Note that this map is indeed a bijection thanks to \Cref{wtYfromXabquo} and \Cref{covWeight}. The bijection $\Omega$ satisfies the condition in 3.1(ii), by giving ordered character triples according to \Cref{CohomPairs}, which in turn via  \Cref{CondiiiProjRep} implies 3.1(ii) according to the Butterfly Theorem \cite[2.16]{SpLausanne} for Brauer characters.
\end{proof}

%\newpage

\section{The inductive BAW condition}
\label{RefomIndBAW}

\subsection{The inductive BAW condition for a set of blocks} We now turn to the version of Alperin's conjecture with blocks, see Conjecture~1.2 above. All blocks are $p$-blocks for a chosen prime $p$.
Definition 3.2 of \cite{KosSpAMBAWCy} provides an inductive BAW condition related to a single block. 

\begin{notation}
For a block $B\in\Bl_p(G)$, $\Alp(B)$ denotes the set of weights $\ol{(Q,\theta)}\in\Alp(G)$ such that $\bl (\theta')^G=B$ for $\theta'\in\Irr(\NNN_G(Q))$ the lift of $\theta$.
Note that this is well defined as $\bl(\theta'^g)^G=\bl(\theta')^G$ for all $g\in G$. We extend the notation to sets $\calB$ of blocks of $G$ by $\Alp(\calB)=\cup_{B\in\calB}\Alp(B)$.
\end{notation}

\begin{rem}\label{CovBlCovWei}
Let $( \wt Q,\wt\theta)$ be a weight of $\wt G$  and $(Q,\theta)$ a weight of $G$ covered by $(\wt Q,\wt\theta)$.
Set $\wt b$ to be the block of $\NNN_{\wt G}(\wt Q)$ which dominates $\bl(\wt\theta)$, and $b$ the block of $\NNN_G(Q)$ which dominates $\bl(\theta)$.
Then by applying \cite[Thm B]{KosSpCliffTh} with the results in \cite[Section 5]{NavSpBrHtZero} it follows that $\wt b^{\wt G}$ covers $b^G$.
\end{rem}

\begin{df}\label{IndBAWCond}
Let $p$ be a prime, $S$ a finite non-abelian simple group, $G$ a $p'$-covering group of $S$ and $B\in \Bl (G)$ such that $Z(G)\cap \ker (\vphi)$ is trivial for any $\vphi\in {\rm IBr}(B)$.
We say that {\bf the inductive BAW condition holds for $B$} if the following statements are satisfied: 
\begin{asslist}
\item There exists an $\Aut(G)_B$-equivariant bijection $$\Omega_B:\IBr(B)\rightarrow \Alp(B).$$ 
\item For every $\varphi\in \IBr(B)$ and $\ol{(Q,\theta)}=\Omega_B(\varphi)$, there exist a finite group $A:=A(\varphi,Q)$ and characters $\wt\varphi\in\IBr(A)$ and $\wt\theta\in\IBr(\NNN_A(Q))$ such that
\begin{enumerate}
\item The group $A$ satisfies $G\lhd A$, $A/{\rm C}_A(G)\cong \Aut(G)_\varphi$, ${\rm C}_A(G)=Z(A)$ and $p\nmid |Z(A)|$,
\item $\wt\varphi\in\IBr(A)$ is an extension of $\varphi$,
\item $\wt\theta\in\IBr(\NNN_A(\ol Q))$ is an extension of $\Omega_B(\varphi)^0$,
\item for every $J$ with $G\leq J\leq A$ the characters $\wt\varphi$ and $\wt\theta$ satisfy
\[\bl(\wt\varphi_J)=\bl (\wt\theta_{\NNN_J(\ol Q)})^J.\]
\end{enumerate}
\end{asslist}
\end{df}

\begin{rem}
As with the inductive Alperin weight condition, an alternative formulation in terms of character triples and then also with projective representations was provided in \cite{SpIndCharTrip}.
In particular, (ii) in \Cref{IndBAWCond} will be verified by showing the following equivalent property
	\[
	\left( ( G\rtimes \Aut(G))_\varphi,G,\varphi\right)\succ_{F,b} \left( ( G \rtimes \Aut(G))_{Q,\theta},\NNN_G(Q),\theta^0\right),
	\]
	see also Theorem~4.4 of \cite{SpIndCharTrip}. 
The order relation on character triples given by ``$ \succ_{F,b}$"  arises from ``$\succ_{F,c}$" with an additional requirement related to \Cref{IndBAWCond}(ii)(4), see Definition~3.2 of \cite{SpIndCharTrip}.
This relation holds if the projective representations $\calP$ and $\calP'$ from \Cref{CondiiiProjRep} can be chosen to satisfy additionally the condition described in Remark~4.5 of \cite{SpLausanne}.

For the verification one can also consider every block $\wh B$ of the universal $p'$-covering group $\wh G$ of $S$. 
Let $\wh B\in\Bl_p(\wh G)$ and  $Z:=\Zent(\wh G)\cap\ker \varphi$ for $\vphi\in\IBr(\wh B)$. 
Then the block $B$ of $G:=\wh G/Z$ contained in $\wh B$ satisfies the inductive BAW condition, if 
\begin{enumerate}
\item there exists an $\Aut(\wh G)_B$-equivariant bijection $\wh \Omega_{\wh B}:\IBr(\wh B)\rightarrow \Alp(\wh B)$, 
\item for every $\wh \phi\in\IBr(\wh B)$, and $\overline{(\wh Q,\wh \theta)}=\wh\Omega_{\wh B}(\phi)$  the associated characters $\phi\in\IBr(\wh G/Z)$ and $\theta \in \Irr(\NNN_{\wh G/Z}(\wh Q) ) $ satisfy $$
	\left( ( G\rtimes \Aut(G))_\varphi,G,\varphi\right)
	\succ_{F,b} \left( ( G \rtimes \Aut(G))_{\wh Q,\theta},\NNN_G(\wh Q),\theta^0\right)$$
	or equivalently  
	$$
		\left( ( \wh G\rtimes \Aut(\wh G))_{\wh \varphi},\wh G,\wh \varphi\right)
		\succ_{F,b} \left( ( \wh G \rtimes \Aut(\wh G))_{\wh Q,\wh \theta},\NNN_{\wh G}(\wh Q),\wh \theta^0\right).$$
		The equivalence of the two relations comes from the fact that $Z$ is a $p'$-group and (the proof of) Corollary~4.5 of \cite{SpDade}.
\end{enumerate}

\end{rem}

\subsection{Criterion for the inductive BAW condition} 
In this section we give a criterion for the inductive BAW condition adapted to simple groups of Lie type. 
It is clear that the last assumption on the orbit stabiliser will not hold for all blocks.
It mimics the criterion provided for the inductive Alperin--McKay condition in \cite{BrSpAMTypeA}.

\begin{thm}\label{NewIndBAWCond}
Let $S$ be a finite non-abelian simple group and $p$ a prime dividing $|S|$.
Let $G$ be a $p'$-covering group of $S$ and $\calB\subseteq \Bl (G)$ a $\wt G$-invariant subset with $(\wt GE)_B\leq (\wt GE)_\calB$ for all $B\in \calB$.
Assume we have a semi-direct product $\wt G \rtimes E$, such that the following conditions hold:
\begin{asslist}
\item 
\begin{itemize}[label={\textbullet}]
\item $G=[\wt G,\wt G]$ and $E$ is abelian or isomorphic to the direct product of a cyclic group and $\mathfrak{S}_3$,
\item $\Cent_{\wt G\rtimes E}(G)=\Zent(\wt G)$ and $\wt GE/\Zent(\wt G)\cong \Aut (G)_{\calB}$ by the natural map,
\item every $\vphi_0\in \IBr(\calB)$ extends to $\wt G_{\vphi_0}$,
\item for every $\overline{(Q,\theta_0)}\in \Alp(\calB)$ the characters $\theta_0$ extends to  $(\NNN_{\wt G}(Q)/Q)_{\theta_0}$.
\end{itemize}

\item Let $\wt \calB= \Bl_p(\wt G\mid \calB)$.
There exists a $\LinBr(\wG/G)\rtimes E$-equivariant bijection
\[\wt \Omega: \IBr(\wt \calB)\longrightarrow \Alp(\wt \calB) \]
with  $\wt \Omega(\wt B)=\Alp(\wt B)$ for all $\wt B\in \wt \calB$ and $\rm J_{G}(\wt\varphi)=\rm J_{G}(\wt \Omega(\wt\varphi))$ for every $\wt\varphi\in \IBr(\calB)$.

\item For every $\wt{\varphi}\in \IBr(\wt \calB)$ there exists some $\varphi\in \IBr (G\mid \wt{\varphi})$ such that
\begin{itemize}[label={\textbullet}]
\item $(\wt G\rtimes E)_{\varphi}=\wt G_{\varphi}\rtimes E_{\varphi}$, and 
\item every $\wG$-conjugate of $\varphi$ extends to its stabilizer in $G\rtimes E$.
\end{itemize}
\item In every $\wt G$-orbit on $\Alp(\calB)$ there exists a weight $(Q,\theta)$ such that 
\begin{itemize}[label={\textbullet}]
\item $(\wG E)_{\ol{(Q,\theta)}}=\wG_{\ol{(Q,\theta)}} E_{\ol{(Q,\theta)}}$, and
\item $\theta$ extends to $(GE)_{Q,\theta}/Q$.
\end{itemize}
\end{asslist}
Then the inductive BAW condition holds for all $p$-blocks $B_0\in \calB$ with abelian $\Out (G)_{\wG\text{-orbit of } B_0}$.
\end{thm}

To prove this we use the following variant of \Cref{CohomPairs} where we keep the group-theoretic assumption of the above statement.
\begin{lm}\label{CohomPairs2}
	Assume that $G$, $\wt G$ and $E$ satisfy the group-theoretic assumptions made in \Cref{NewIndBAWCond}. Let $\varphi\in\IBr(G)$ and $\ol{(Q,\theta)}\in\Alp(G)$ such that the following holds:
	\begin{enumerate}[(i)]
		\item $(\wt G\rtimes E)_{\varphi}=\wt G_{\varphi}\rtimes E_{\varphi}$, 
		$[(\wt G\rtimes E)_{\varphi}, \wG]\leq \wG_\varphi$,
		 and all $\wG$-conjugates of $\varphi$ extend in $G\rtimes E_\varphi$ and $\wt G_\varphi$.
		\item $(\wG E)_{\ol{(Q,\theta)}}=\wG_{\ol{(Q,\theta)}}E_{\ol{(Q,\theta)}}$ and $\theta$ extends to $(G\rtimes E)_{Q,\theta}/Q$ and to $\NNN_{\wt G}(Q)_\theta/Q$.
		\item $(\wG E)_\varphi=(\wG E)_{\overline{(Q,\theta)}}$.
		\item There exist $\wt \varphi\in\IBr(\wG \mid \varphi)$ and $\olwtQwttheta\in\Alp(\wG\mid \overline{(Q,\theta)})$ such that for every $x\in E_{\varphi}$  there exists some $\mu\in\IBr(\wt G/G)$ with $\mu\cdot\olwtQwttheta=\olwtQwttheta ^x$,  $\wt \varphi^x=\wt\varphi\mu$ and $\bl(\wt \varphi)=\bl(\wt \theta)^{\wt G}$.
	\end{enumerate}
	Let $Z:=\ker(\varphi)\cap \ZZ(G)$. Then
	\[
	\left( (\wt GE)_\varphi/Z,G/Z,\ol\varphi'\right)\succ_{F,b} \left( (\wt GE)_{Q,\theta}/Z,\NNN_G(Q)/Z,\theta^0\right),
	\]
	for some $\wt G$-conjugate $\varphi'$ of $\varphi$. 
	\begin{proof}
Note that as $\bl(\wt \varphi)=\bl(\wt\theta)^{\wt G}$, it follows that $\wt \varphi$ and $\wt \theta$ lie above the same central character.
Because of $[(\wt G\rtimes E)_{\varphi}, \wG]\leq \wG_\varphi$ all $\wG$-conjugates have $(\wG\times E)_\varphi$ as their stabilizer in $\wG\rtimes E$. Thus by \Cref{CohomPairs}, it follows that 	
	\[
	\left( (\wt GE)_\varphi/Z,G/Z,\ol\varphi'\right)\succ_{F,c} \left( (\wt GE)_{Q,\theta}/Z,\NNN_G(Q)/Z,\theta^0 \right)
	\]
	for all $\wG$-conjugates $\varphi'$ of $\varphi$. (Here $\overline{ \varphi'}$ denotes the character of $G/Z$ associated to $\varphi'$.) 
	
	Let $(\wt Q, \wt \theta)$ a representative of the weight $\overline{(\wt Q,\wt \theta)}$ and $\wh \theta$ be an extension of $\theta$ to $\NNN_\wG(Q)_\theta/Q$ such that via induction and a map $\Delta_\theta$ it corresponds to $(\wt Q, \wt \theta)$. Let $\varphi'$ be some $\wG$-conjugate of $\varphi$ such that some extension $\wh\varphi'$ of $\varphi'$ to $\wG_{\varphi}$ satisfies $(\wh\varphi')^\wG=\wt\varphi$ and $\bl(\wh\varphi)=\bl(\wh\theta)^{\wG_{\varphi}}$, where $\wh \theta$ is an extension of $\theta$ to $\NNN_\wG(Q)_\theta/Q$. The existence of $\wh \theta$ and $\wh \varphi'$ follows from Clifford theory and the facts that $\wt G/G$ is abelian and $\theta$ extends to $\NNN_{\wt G}(Q)/Q$ by (ii).
	
	This implies
	\[
		\left( \wt G_\varphi/Z,G/Z,\ol\varphi'\right)\succ_{F,b} \left( \wt G_{Q,\theta}/Z,\NNN_G(Q)/Z,\theta^0 \right)
		\]
	for this particular character $\varphi'$. 
	According to \cite[Lem.~3.2]{CabSpAMTypeA}, see also its application in the proof of  \cite[Thm.~2.4]{BrSpAMTypeA},
	this implies 
		\[
			\left( (\wt GE)_\varphi/Z,G/Z,\ol\varphi'\right)\succ_{F,b} \left( (\wt GE)_{Q,\theta}/Z,\NNN_G(Q)/Z,\theta^0 \right)
			\]
	for the case of an abelian $E$. 
	
	If $E$ is non-abelian, then it is isomorphic to $\mathfrak S_3\times C$ for some finite cyclic group $C$.
	When the considerations above are applied to $\wh \varphi'$ and $\wh \theta$, Lemma~3.2 of  \cite{CabSpAMTypeA} is replaced by the result  \cite[Lem.~6.13]{Ruh_Jordan_decomp} for Brauer characters. Note that in the proof of \cite[Lem.~6.13]{Ruh_Jordan_decomp} ordinary characters can be replaced by Brauer characters without additional requirements. Again we obtain
	\begin{align*}
	\left( (\wt GE)_\varphi/Z,G/Z,\ol\varphi'\right)& \succ_{F,b} \left( (\wt GE)_{Q,\theta}/Z, \NNN_G(Q)/Z, \theta^0 \right). 
	\qedhere
	\end{align*}
	\end{proof}

\end{lm}

\begin{proof}[Proof of \Cref{NewIndBAWCond}]
	Let $\wt \calB$ be the set of blocks of $\wG$ covering some block of $\calB$. Then the set $\Alp(\wt \calB)$ consist of all the weights of $\wG$ covering some weight of $\Alp(\calB)$ according to \Cref{CovBlCovWei}, and this set is $E':=E_\calB$-stable. We can choose $\calA$ to be a $\LinBr(\wG/G)\rtimes E'$-transversal in $\Alp(\wt \calB)$.   As in the proof of \Cref{AWStCond} we choose $\calA_0$ to be a set of weights such that each weight of $G$ in $\calA_0$ is covered by exactly one of $\calA$ and it satisfies \ref{NewIndBAWCond}(iv). 
	We see that $\calA_0$ is then a $\wt G\rtimes E'$-transversal in $\Alp(\calB)$.

	Like in the proof of \Cref{AWStCond} let $\calG:=\wt\Omega^{-1}(\calA)$. We see that it is a $\LinBr(\wt G/G)\rtimes E'$-transversal. 
	
	Let $\wt \vphi\in\calG$. Then there exist $\overline{(Q,\theta)}\in\calA_0$ and  $\overline{(\wt Q,\wt \theta)}\in\calA$ such that $\wt\Omega(\wt\varphi)=\overline{(\wt Q,\wt \theta)}$ and  $\overline{(\wt Q,\wt \theta)}$ covers $\overline{(Q,\theta)}$. 
	As $\Out(G)_{\wG-\text{-orbit of }B_0}$ is abelian, $[(\wG\rtimes E)_\vphi,\wG]\leq \wG_\vphi$ holds for every $\vphi\in\IBr(B_0)$. Hence we can apply \Cref{CohomPairs2} and obtain that there is some $\vphi\in\IBr(G\mid \wt \vphi)$ such that 
		\[
			\left( (\wt GE)_\varphi/Z,G/Z,\ol\varphi\right)\succ_{F,b} \left( (\wt GE)_{Q,\theta}/Z,\NNN_G(Q)/Z,\theta^0 \right). 
			\]
	For each $\wt\vphi\in\calG$ we choose one character $\vphi$ that way. Let $\calG_0$ be the set of characters $\vphi$ thus obtained. 
	
	We define $\Omega$ first on $\calG_0$: 
	for $\vphi\in\calG_0$ one fixes $\Omega(\vphi)$ to be the unique element in $\calA_0$ such that $\wt \Omega(\IBr(\wt G\mid \vphi))=\Alp(\wt G\mid \Omega(\vphi))$. 
	As in the proof of \Cref{AWStCond} the stabilizers of $\vphi$ and $\Omega(\vphi)$ in $\wt G\rtimes E'$ coincide and $\Omega$ is well-defined as a $\wG\rtimes E'$-equivariant bijection between $\IBr(\calB)$ and $\Alp(\calB)$ giving ordered character triples. Since $\succ_{F,b}$ is preserved via conjugation, $\Omega$ has all the required properties. 
\end{proof}

In view of the possible applications the requirement \ref{NewIndBAWCond}(iii) can be lightened, by only requiring that $\vphi$ extends to its stabilizer in $G\rtimes E$. When for $G$, $\wG$ and $E$ certain groups are chosen then it is sufficient to require that one character extends to its stabilizer in $G\rtimes E$. This aspect was implicitly used in the proof of \cite[Thm.~2.4]{BrSpAMTypeA}. 

\begin{lm} 
Let $\bG$ be some simply connected simple algebraic group $\bG$ with an $\FF_q$-structure (where $q$ is a power of a prime that is not assumed to be $p$) given by a Frobenius endomorphism $F\colon \bG\to \bG$ and $\wt \bG$ an algebraic group coming from a regular embedding of $\bG$ and acted on by $F$. Let $E\leq \Aut(\GF)$ be the subgroup induced by field and graph automorphisms of $\widetilde{\bG}$ such that $\widetilde{\bG}^F\rtimes E$ induces the whole automorphism group of $\GF$, (see \cite[2.5.12]{GLS3} for the structure of ${\rm Aut}(\GF)$).

If $\calB$ is a $\wGF$-orbit in $\Bl_p(\GF)$ with abelian $\Out(\GF)_\calB$ and $\varphi\in\Irr(\calB)$ satisfies Condition \ref{IndBAWCond}(iii), then $\varphi^g$  satisfies Condition \ref{IndBAWCond}(iii) for any $g\in \widetilde{\bG}^F$.
\end{lm}
\begin{proof} 
Let $G:=\GF$ and $\wG:=\wG^F$.
Since ${\rm Out}(G)_{\mathcal{B}}$ is abelian, $\varphi':=\varphi^g$ satisfies $$(\widetilde{G}\rtimes E)_{\varphi}= \widetilde{G}_{\varphi'}\rtimes E_{\varphi'}.$$
We have to show that $\varphi'$ extends to $GE_{\varphi'}=GE_\varphi$.

Without loss of generality it can be assumed that $E_\varphi$ is non-cyclic and $\wt G_\varphi\neq \wt G$, in particular $Z(G)\neq 1$.
This implies that $G$ is of type A$_n(q)$, D$_n(q)$ or E$_6(q)$ (all untwisted).
Let $\overline{\gamma}, \overline{F_0}\in \Out(G)$ be associated with a graph automorphism $\gamma$ of $G$ of order $2$ and with a  field automorphism $F_0$, such that $E=\langle {F_0},{\gamma}\rangle$. 
Since $E$ is non-cyclic, ${F_0}$ is of even order (and so $4\mid (q-1)$, if $q$ is odd).

In type D$_{2m}$, field automorphisms centralise ${\rm Diag}(G):= \wG/(G.Z (\wG))$ while the graph automorphisms acts faithfully on it.
Since ${\rm Out}(G)_\mathcal{B}$ is abelian, $E_{\calB}$ only contains field automorphisms and thus is cyclic.
In types A$_{n-1}$, D$_{2m+1}$, E$_6$ the automorphism $\gamma$ acts by inversion on ${\rm Diag}(G)$.
As $E_\varphi$ is non-cyclic, it follows that $E_{\varphi}=\langle F_0^i,{\gamma}\rangle\leq E_\calB$ for some positive integer $i$ and so $Z(G)\cong {\rm Diag}(G)\cong C_2$.
This excludes type E$_6$ where $|Z(G)|= 3$, also D$_{2m+1}$ because then $|Z(G)|={\rm gcd}(4,q-1)$ and type A$_{n-1}$ with $n \not\equiv 2 \mod 4$ because then $|Z(G)|={\rm gcd}(n,q-1)$.
Hence the remaining case is when $G$ is of type A$_{n-1}$ with $n\equiv 2\mod 4$ and $4\mid q-1$.

Let $\mathcal L_F$ be the Lang map  corresponding to $F$ with $G=\bfG ^F$, which commutes with $F_0$ and $\gamma$.
Additionally $G\lhd \wh G:=\mathcal{L}_F^{-1}(Z(\bfG))$. 
Since $F_0$ and $\gamma$ come from endomorphisms of $\bG$, there are automorphisms $\wh{F_0}$ and $\wh{\gamma}$ of $\wh G$ extending the automorphisms $F_0$ and $\gamma$ of $G$.
From the equality $\wt \bG= \bG Z(\wt \bG)$ and Lang's theorem, it follows that there is an isomorphism $$\widetilde G/G Z(\widetilde G) \cong \wh G/ G Z({\bf G})$$ which preserves the conjugation action on $G$.
Moreover, applying the Lang map yields isomorphisms 
$$ \wh G/G \cong  Z({\bfG}) \text{ and } \wh G/ G Z({\bf G})\cong Z(\bfG)/\mathcal{L}_F(Z(\bfG)).$$

As $G=\SL_n(q)$ with $4|q-1$ and $n\equiv 2 \mod 4$, we observe that $\mathcal{L}_F(Z(\bfG))$ does not contain an element of order $2$, since $\mathcal{L}_F$ acts on $Z(\bfG)$ by $x\mapsto x^{q-1}$. 
Let $x\in \wh G\setminus GZ(\bfG)$.
Then $\mathcal L_F(x)$ is the central involution of $G$ and so $\wh{F_0}(x)x^{-1}, \wh\gamma(x)x^{-1}\in G$.
Moreover $x$ induces the non-trivial diagonal automorphism of $G$ and so $\varphi^g=\varphi^x$ is the only character of $G$ that is $\wG$-conjugate to $\varphi$  but different from $\varphi$.
 
It follows that $\varphi^x$ extends to $G\rtimes E_{\varphi}$ if and only if $\varphi^x$ extends to $G\langle  \overline{\gamma}\rangle$ and some extension is $E_{\varphi}$-invariant.

Take $\widetilde{\varphi}$ an ${F_0}^i$-invariant extension of $\varphi$ to $G\langle \gamma\rangle$. 
Then $\widetilde{\varphi}^x$ is an extension of $\varphi^x$ to $G\rtimes \langle {\gamma}\rangle\cong G\rtimes \langle \wh{\gamma}\rangle$ since $\wh{\gamma}^x\in G \wh{\gamma}$. By definition $\widetilde{\varphi}^x$ is $(\wh{F_0}^i)^x$-invariant. Since $(\wh{F_0}^i)^x\in G\wh{F_0}^i$ it is also $\wh{F_0}^i$-invariant and thus ${F_0}^i$-invariant.
This implies that $\varphi^x$ extends to $GE_{\varphi}=GE_{\varphi^x}$.
\end{proof}

\bigskip
\noindent
{ J. B. and B. S.: \address{School of Mathematics and Natural Sciences University of Wuppertal, Gau\ss str. 20, 42119 Wuppertal, Germany}}
\end{document}